\documentclass[a4paper,11pt]{amsart}

\usepackage[applemac]{inputenc}
\usepackage{amsmath, amsthm, amssymb}
\usepackage[all]{xy}
\usepackage{marginnote}
\usepackage{color}

\theoremstyle{plain}
\newtheorem{theorem}{Theorem}[section]

\newtheorem{proposition}[theorem]{Proposition}
\newtheorem{corollary}[theorem]{Corollary}
\newtheorem{lemma}[theorem]{Lemma}

\newtheorem{question}[theorem]{Question}
\theoremstyle{definition}
\newtheorem{definition}[theorem]{Definition}
\newtheorem{example}[theorem]{Example}
\theoremstyle{remark}
\newtheorem{remark}[theorem]{Remark}
\newtheorem{claim}[theorem]{Claim}

\begin{document}
\bibliographystyle{alpha}

\title[Exceptional set and the Green--Griffiths locus]{The exceptional set and the Green--Griffiths locus do not always coincide} 

\author{Simone Diverio \and Erwan Rousseau}
\address{Simone Diverio \\ CNRS et Institut de Math\'ematiques de Jussieu -- Paris Rive Gauche.}
\email{diverio@math.jussieu.fr} 
\address{Erwan Rousseau \\ Aix Marseille Universit\'e, CNRS, Centrale Marseille, I2M, UMR 7373, 13453 Marseille, France.}
\email{erwan.rousseau@univ-amu.fr}

\thanks{Both authors were partially supported by the ANR project \lq\lq POSITIVE\rq\rq{}, ANR-2010-BLAN-0119-01.}
\date{\today}

\begin{abstract}
We give a very simple criterion in order to ensure that the Green--Griffiths locus of a projective manifold is the whole manifold. Next, we use it to show that the Green--Griffiths locus of any projective manifold uniformized by a bounded symmetric domain of rank greater than one is the whole manifold. In particular, this clarifies an old example given by M. Green to S. Lang.
\end{abstract}

\maketitle

\section{introduction}

Let $X$ be a complex projective manifold, $V\subset T_X$ a rank $r$ holomorphic subbundle, and call
$$
E_{k,m}^{GG}V^*\to X
$$
the holomorphic vector bundle of jet differentials of order $k$ and weighted degree $m$ acting on germs of holomorphic curves tangent to $V$. For definitions and basic properties of these bundles we refer to \cite{GG80} and \cite{Dem97}. Jet differentials are particularly useful in connection to hyperbolicity--related problems thanks to the following.

\begin{theorem}[Green--Griffiths \cite{GG80}, Siu--Yeung \cite{SY97}, and Demailly \cite{Dem97}]\label{fund}
Let $(X,V)$ be a directed projective manifold and $A\to X$ an ample line bundle. Then, for any entire curve $f\colon\mathbb C\to X$ tangent to $V$ and any $P\in H^0(X,E_{k,m}^{GG}V^*\otimes A^{-1})$, one has $P(f', f'', \dots,f^{(k)})\equiv 0$. 
\end{theorem}

Now, fix an ample line bundle $A\to X$. Define $GG_A(X,V)$ to be the set of points $x\in X$ such that for all integers $k>0$ there exists a $k$-jet of holomorphic curve $\varphi_k\colon(\mathbb C,0)\to (X,x)$ tangent to $V$ with the property that for all integers $m>0$, every global jet differential of order $k$ and weighted degree $m$ with values in $A^{-1}$ vanishes whenever evaluated on (the $k$-jet defined by) $\varphi_k$ at $0$. Thus, if $f\colon\mathbb C\to X$ is an entire curve tangent to $V$, Theorem \ref{fund} tells us immediately that $f(\mathbb C)\subset GG_A(X,V)$. We shall see in the next section that the locus $GG_A(X,V)$ is indeed independent of the particular ample line bundle chosen.

\begin{definition}
Let $X$ be a projective manifold and $V\subset T_X$ a holomorphic subbundle. The \emph{exceptional set} $\operatorname{Exc}(X,V)$ of $(X,V)$ is defined to be the Zariski closure of the union of all the images of entire curves traced in $X$ and tangent to $V$. 
The \emph{Green--Griffiths locus $GG(X,V)$ of $(X,V)$} is defined as
$$
GG(X,V)=GG_A(X,V),
$$
for some (and hence any) ample line bundle $A\to X$.

In the \emph{absolute case $V=T_X$}, we shall simply call $\operatorname{Exc}(X,T_X)=\operatorname{Exc}(X)$ and $GG(X,T_X)=GG(X)$.
\end{definition}

By Theorem \ref{fund}, one always has the inclusion:
$$
\operatorname{Exc}(X,V)\subset GG(X,V).
$$

It is conjectured in \cite{GG80} that if $X$ is a $n$ dimensional projective manifold of general type and $V=T_X$, then there should exists a large integer $k=k(X)$ such that the growth rate of 
$$
m\mapsto \dim H^0(X,E_{k,m}^{GG}T_X^*)
$$ 
would have to be maximal, that is asymptotic with $m^{(k+1)n-1}$. This conjecture was proven in \cite{GG80} for projective surfaces and therein supported in all dimension by an Euler characteristic computation. It was then established in the special case of smooth projective hypersurfaces of general type in \cite{Mer10} and was finally proven in full generality only very recently by J.--P. Demailly in \cite{Dem11} by means of his holomorphic Morse inequalities combined with a delicate \lq\lq probabilistic\rq\rq{} curvature estimate. In particular, if $X$ is a projective manifold of general type and $A\to X$ an ample line bundle, then there always exists large integers $k,m>0$ such that $H^0(X,E_{k,m}^{GG}T^*_X\otimes A^{-1})\ne\{0\}$ and indeed this space is very big. It is thus legitimate to ask if much more is true, namely that if $X$ is a projective manifold of general type, then the Green--Griffiths locus is always a proper subvariety of $X$, \textsl{i.e.}
$$
GG(X)\subsetneq X.
$$
An affirmative answer to this would lead directly to the solution of the so--called Green--Griffiths--Lang conjecture claiming the algebraic degeneracy of entire curves in projective manifolds of general type. One could even speculate more and ask whether $\operatorname{Exc(X)}=GG(X)$, that is whether $GG(X)$ provides an algebraic description of the exceptional set. About this, let us quote S. Lang from his paper \cite{Lan86}, page 200:

\smallskip
\noindent
\emph{\lq\lq In particular, the exceptional set [...] is contained in the 
Green--Griffiths set [...]. I asked Green--Griffiths whether they might be equal. 
Green told me that the two sets are not equal in general. Certain Hilbert 
modular surfaces constructed by Shavel \cite{Sha78}, compact quotients of the product 
of the upper half--plane with itself, provide a counterexample which is hyperbolic, but such that the Green--Griffiths set is the whole variety. Thus the jet 
construction appears insufficient so far to characterize the exceptional set [...] completely algebraically.\rq\rq{}}
\smallskip

\noindent
In \cite{Lan86} there are no more details about that, nor in the subsequent literature as far as we know: we have somehow the impression that this paragraph has been passed by unnoticed. Here, we take the opportunity to give a detailed account of the above key fact by proving indeed a much more general result, as well as several consequences.

Let $\mathcal W$ be a coherent analytic sheaf on a complex manifold $X$ and $m>0$ be an integer. We shall denote by $\mathcal W^{[m]}$ the double dual of the $m$-th tensor power $\mathcal W^{\otimes m}$ of $\mathcal W$. The starting point is the following criterion.

\begin{theorem}\label{main}
Let $(X,V)$ be a complex projective directed manifold. Suppose that there exists a saturated coherent analytic subsheaf $\mathcal W\subset\mathcal O_X(V)$ such that for some ample line bundle $A\to X$ one has
$$
H^0(X,(\mathcal W^*)^{[m]}\otimes A^{-1})=0,\quad\forall m\ge 1.
$$
Then, $GG(X,V)=X$.
\end{theorem}

Observe that we do not require any Frobenius integrability property for $\mathcal W$. Observe moreover, that by a very recent result of \cite{CP13}, in the case where $\mathcal W=T_X$, the condition in the statement is satisfied if and only if $X$ is not of general type. We now pass to direct and less direct consequences of the criterion above. 

First of all, we have the following result. It follows straightforwardly from Theorem \ref{main} and a techincal elementary lemma about growth rate of sections of tensor powers of pull--back of vector bundles (see Section \ref{sectionproduct} for more details and in fact a slightly more general result concerning finite surjective morphisms).

\begin{corollary}\label{introproduct}
Let $X\simeq Y\times Z$ be a complex projective manifold isomorphic to a non trivial product. Then, $GG(X)=X$.
\end{corollary}

This already shows that in order to produce a counterexample in the spirit of Green--Lang it is not necessary to look so far away at Hilbert modular surfaces, but it is indeed sufficient to consider the product of two hyperbolic compact Riemann surfaces. Here, it is worthwhile mentioning also that a slight variation of the corollary above permits to show that certain surfaces constructed in \cite{BCGP12} provide examples of simply connected surfaces of general type for which the Green--Griffiths locus is the entire surface but the exceptional locus is not known. Therefore, there exist surfaces whose hyperbolicity properties are unknown, and for which  it is not possible to use the jet differentials techniques alone in order to deduce such properties (see Section \ref{sectionproduct} for more details about these examples).

Next, let $\Omega\subset\mathbb C^n$ be a bounded symmetric domain, $\Gamma \subset\operatorname{Aut}(\Omega)$ a cocompact torsion free lattice and $X=\Omega/\Gamma$. The Bergman metric of $\Omega$ is $\operatorname{Aut}(\Omega)$-invariant and descends to a metric $\omega$ on every smooth quotient $X$. Moreover, the cohomology class of $\omega$ is $-2\pi\,c_1(X)$. Thus, $K_X$ is positive and in particular $X$ is projective and of general type. Moreover, whenever $X$ is a complex space whose universal cover is a bounded domain in $\mathbb C^n$, then every entire curve must be constant by Liouville's theorem. Hence, the exceptional locus $\operatorname{Exc}(X)$ is empty. If $\Omega$ has rank one or, equivalently, if it is isomorphic to the unit ball $\Omega\simeq\mathbb B^n$, then it is well known that $X$ has ample cotangent bundle. In particular, already by looking at order one jet differentials one finds that $GG(X)=\emptyset$. The situation changes drastically as soon as the rank of $\Omega$ becomes bigger than one.

\begin{theorem}\label{bsd}
Let $X$ be a projective manifold uniformized by a bounded symmetric domain $\Omega$. Then, either $\Omega\simeq\mathbb B^n$ and thus $GG(X)=\emptyset$, or $GG(X)=X$.
\end{theorem}

This result should be compared to the \lq\lq all or nothing\rq\rq{} principle on the Lang locus of rational points in Shimura varieties established in \cite{UY10} (see also Section \ref{UY} for more details about this correspondence).

Theorem \ref{bsd} provides a very large class of projective manifolds of general type for which the worst possible situation occurs, as far as a description of the exceptional locus in terms of the Green--Griffiths locus is concerned:
$$
\emptyset =\operatorname{Exc}(X)\subsetneq GG(X)=X.
$$
This includes of course also the particular Hilbert modular surfaces above mentioned. 

Among these locally symmetric manifolds, it is worth distinguishing the two subclasses of locally reducible and of locally irreducible ones. The proof of Theorem \ref{bsd} will be quite different in the two cases. In fact, while in the locally reducible case it is immediate to identify (after a finite unramified covering) a foliation ---and thus a subsheaf of the tangent sheaf--- to work with in order to apply directly Theorem \ref{main}, in the locally irreducible case there is no natural foliations at our disposal. To overcome this difficulty, one possibility in the case of rank at least $2$ is to use the arithmeticity of the lattice and the theory of Shimura varieties (see beginning of Section \ref{locirred} for details). Alternatively, it is possible to avoid the use of the arithmeticity of the lattice using the theory of characteristic bundles \cite{Mok89}. This shall provide a holomorphic fiber bundle over $X$, endowed with a holomorphic foliation by curves given by liftings of minimal disks (we are grateful to N. Mok for suggesting us several things about this approach). This foliation will turn out to be of negative Kodaira dimension and this will be sufficient to obtain the desired conclusion (see Section \ref{locirred} for the details).

\subsubsection*{Acknowledgments}

The first--named author would like to warmly thank Nicolas Bergeron, Olivier Biquard, S\'ebastien Boucksom, Fabrizio Catanese, Gilles Courtois, Jean--Pierre Demailly, Antonio J. Di Scala, Bruno Klingler and Roberto Pignatelli for very interesting discussions about several themes concerning different aspects of this paper.

The second--named author thanks Ziyang Gao, Carlo Gasbarri, Steven Lu, Michael McQuillan and Ngaiming Mok for stimulating discussions. He also thanks CNRS for the opportunity to spend a semester in Montreal at the UMI CNRS-CRM and the hospitality of UQAM-Cirget where part of this work was done.

\section{Proof of Theorem \ref{main}}\label{2}

Let us start by recalling some basic facts from \cite{Dem97} about jet differentials in the general framework of directed manifolds (for more details see the cited references). 

So, let $(X,V)$ be a directed manifold, that is a complex manifold $X$ together with a holomorphic subbundle $V\subset T_X$, non necessarily integrable, of the holomorphic tangent bundle $T_X$ of $X$. We call $J_kV$ the bundle of $k$-jets of holomorphic curves $\varphi\colon(\mathbb C,0)\to X$ which are tangent to $V$, that is $\varphi'(t)\in V_{\varphi(t)}$ for all $t$, together with the projection $\varphi\mapsto \varphi(0)$ onto $X$. It is a holomorphic fiber bundle, which is naturally a subbundle of the bundle $J_kT_X$ of $k$-jets of holomorphic curves with values in $X$ with no further restrictions. Moreover, there is a canonically defined fiber--wise $\mathbb C^*$-action on $J_kV$ given by the reparametrization of a $k$-jet tangent to $V$ by the homotheties corresponding to elements of $\mathbb C^*$. Of course, this action is compatible with restriction to smaller subbundles.

Let $E_{k,m}^{GG}V^*\to X$ be the holomorphic vector bundle over $X$ whose fibers are complex valued polynomials on the fibers of $J_kV$ of weighted degree $m$ with respect to the above--defined $\mathbb C^*$-action. If $W\subset V\subset T_X$ are holomorphic subbundles, the inclusions
$$
J_kW\subset J_k V\subset J_k T_X
$$
induce surjective arrows
$$
E_{k,m}^{GG}T^*_X\to E_{k,m}^{GG}V^*\to E_{k,m}^{GG}W^*,
$$
so that $E_{k,m}^{GG}W^*$ can be regarded as a quotient of $E_{k,m}^{GG}V^*$, the quotient map being given by evaluating jet differentials tangent to $V$ only on $k$-jets tangent to $W$. 

\begin{remark}\label{vanishing}
In particular, we see that if every global section of $E_{k,m}^{GG}W^*$ vanishes at some point $j_k\in J_kW\subset J_kV$, then so every global section of $E_{k,m}^{GG}V^*$ does.
\end{remark}

Now, considering the highest monomials (with respect to the reverse lexicographic order) gives a natural filtration on weighted homogeneous polynomials. Such filtration defines an intrinsic filtration on the bundles $E_{k,m}^{GG}V^*$, whose graded series are given by
$$
\operatorname{Gr}^{\bullet}E_{k,m}^{GG}V^*=\bigoplus_{\ell_1+2\ell_2+\cdots+k\ell_k=m}S^{\ell_1}V^*\otimes\cdots\otimes S^{\ell_k}V^*.
$$
Therefore, it follows that if for all $k$-tuple of non negative integers $(\ell_1,\dots,\ell_k)$ such that $\ell_1+2\ell_2+\cdots+k\ell_k=m$, we have
$$
H^0(X,S^{\ell_1}V^*\otimes\cdots\otimes S^{\ell_k}V^*)=0,
$$
then
$$
H^0(X,E_{k,m}^{GG}V^*)=0,
$$
as well. 

Before entering into the proof of Theorem \ref{main}, let us show as promised that the locus $GG_A(X,V)$ is independent of the ample line bundle $A\to X$.

\begin{lemma}\label{ind}
The set $GG_A(X,V)$ does not depend on the ample line bundle $A\to X$.
\end{lemma}

\begin{proof}
Let $A,B\to X$ two ample line bundles and $\ell>0$ be a positive integer such that $B^{\otimes\ell}\otimes A^{-1}$ is globally generated. We shall show that 
$$
GG_A(X,V)\subset GG_B(X,V),
$$
and the equality will follow by interchanging the roles of $A$ and $B$. 

Let $x\not\in GG_B(X,V)$. Then, there exist integers $k,m>0$, a global section $P\in H^0(X,E_{k,m}^{GG}V^*\otimes B^{-1})$ and a $k$-jet of holomorphic curve $\varphi_k\colon(\mathbb C,0)\to (X,x)$ tangent to $V$ such that $P(\varphi_k)\ne 0$. Let $\sigma\in H^0(X,B^{\otimes\ell}\otimes A^{-1})$ be a global section such that $\sigma(x)\ne 0$. Then, 
$$
P^\ell\otimes\sigma\in H^0(X,E_{k,\ell m}^{GG}V^*\otimes A^{-1})
$$
and 
$$
(P^\ell\otimes\sigma)(\varphi_k)=P^\ell(\varphi_k)\cdot\sigma(x)\ne 0,
$$
so that $x\not\in GG_A(X,V)$.
\end{proof}

\begin{proof}[Proof of Theorem \ref{main}]
Let us first suppose that $\mathcal W=\mathcal O_X(W)$ is the sheaf of germs of holomorphic sections of a holomorphic subbundle $W\subset V$. Let $A\to X$ be any ample line bundle and $k,m>0$ be two integers. For all $k$-tuple of non-negative integers $(\ell_1,\dots,\ell_k)$ such that $\ell_1+2\ell_2+\cdots+k\ell_k=m$ we have 
$$
S^{\ell_1}W^*\otimes\cdots\otimes S^{\ell_k}W^*\otimes A^{-1}\subset (W^*)^{\otimes |\ell|}\otimes A^{-1},
$$ 
where $|\ell|=\sum\ell_j$. Thus, 
$$
H^0(X,S^{\ell_1}W^*\otimes\cdots\otimes S^{\ell_k}W^*\otimes A^{-1})=0
$$
and, since taking the tensor product of $E_{k,m}W^*$ with $A^{-1}$ affects the corresponding graded bundle just tensoring by $A^{-1}$, we have
$$
H^0(X,E_{k,m}^{GG}W^*\otimes A^{-1})=0.
$$
Moreover, we have of course surjective morphisms
$$
E_{k,m}^{GG}V^*\otimes A^{-1}\to E_{k,m}^{GG}W^*\otimes A^{-1}.
$$
Therefore, by Remark \ref{vanishing}, in such a situation every global jet differential tangent to $V$ of order $k$ and weighted degree $m$ with values in $A^{-1}$ must vanish when evaluated on $k$-jets which are tangent to $W$. Since at each point $x\in X$ we have such jets, it follows that $GG(X,V)=X$.

Now, take $\mathcal W$ as in the hypotheses. Since $\mathcal W$ is saturated, $\mathcal O_X(V)/\mathcal W$ is torsion-free and thus locally free in codimension two. In other words, there exists a proper subvariety $Y\subsetneq X$, $\operatorname{codim}_X Y\ge 2$, and a holomorphic vector bundle 
$$
W\to U:=X\setminus Y
$$ 
on the dense open subset $U$ such that $\mathcal O_U(W)\simeq\mathcal W|_U$ and $W$ is actually a subbundle of the restriction $V|_{U}$ of $V$ to $U$. Of course, over $U$, we also have $\mathcal W^*|_U\simeq\mathcal O_U(W^*)$. Since $(\mathcal W^*)^{[|\ell|]}$ is reflexive by definition and thus normal, we have a surjection
$$
H^0(X,(\mathcal W^*)^{[|\ell|]}\otimes A^{-1})\to H^0(U,(\mathcal W^*)^{[|\ell|]}\otimes A^{-1})=H^0(U,(W^*)^{\otimes|\ell|}\otimes A^{-1}),
$$
so that $H^0(U,(W^*)^{\otimes|\ell|}\otimes A^{-1})=\{0\}$. From
$$
S^{\ell_1}W^*\otimes\cdots\otimes S^{\ell_k}W^*\otimes A^{-1}|_U\subset (W^*)^{\otimes |\ell|}\otimes A^{-1}|_U
,$$
we conclude that 
$$
H^0(U,\operatorname{Gr}^{\bullet}E_{k,m}^{GG}W^*\otimes A^{-1})=\{0\}
$$
and thus $H^0(U,E_{k,m}^{GG}W^*\otimes A^{-1})=\{0\}$. Then, since each section of $E_{k,m}^{GG}V^*\otimes A^{-1}$ over $U$ extends to a section of $H^0(X,E_{k,m}^{GG}V^*\otimes A^{-1})$, we have $U\subset GG(X,V)$ and since $GG(X,V)$ is a closed set, then $GG(X,V)=X$.
\end{proof}

\subsection{Theorem \ref{main} for Demailly--Semple jets}

In this subsection we would like to give a somehow more geometric proof of (a special case of) Theorem \ref{main} in the case of invariant jet differentials, even if this case is of course included in what is proved above. For the sake of simplicity we shall assume that $L\subset T_X$ is a holomorphic line subbundle and not just an injection of sheaves of arbitrary rank. In this case, the hypothesis on $L$ simply become that $L$ is not big, that is its Kodaira--Iitaka dimension is not maximal. Invariant jet differentials were introduced in \cite{Dem97} as a refined version of Green--Griffiths' ones: they are constructed taking invariant holomorphic functions on the $k$-jet space by a larger group, namely the full group of $k$-jets of biholomorphisms of $(\mathbb C,0)$, instead of merely homotheties. The vector bundle $E_{k,m}V^*\to X$ of invariant jet differential of order $k$ and weighted degree $m$ acting on germs of holomorphic curves tangent to $V\subset T_X$ can be obtained as a direct image of an invertible sheaf on a tower of projective spaces, as follows.

Start with a directed manifold $(X,V)$ and define the new directed manifold $(\widetilde X,\widetilde V)$ to be
$$
\widetilde X=P(V),\quad\widetilde V=\pi_*^{-1}\bigl(\mathcal O_{P(V)}(-1)\bigr)\subset T_{P(V)},
$$
where $\pi\colon P(V)\to X$ is the projectivized bundle of lines of $V$ and $\mathcal O_{P(V)}(-1)$, which is in turn a subbundle of $\pi^* V$, the tautological line bundle of $P(V)$.
Now, set $(P_0(V),V_0)=(X,V)$ and define inductively 
$$
(P_k(V),V_k)=(\widetilde{P_{k-1}(V)},\widetilde{V_{k-1}})
$$ 
together with the total projection $\pi_{0,k}\colon P_k(V)\to X$. The functoriality of this construction shows that if $(Y,W)\subset (X,V)$ is a directed submanifold (\textsl{i.e.} $Y$ is a smooth submanifold of $X$, possibly the whole $X$, and $W\subset T_Y\subset T_X|_{Y}$ is a holomorphic subbundle of $V|_Y$), then for all positive integers $k$ we have
$$
P_k(W)\subset P_k(V),\quad W_k\subset V_k|_{P_k(W)}\quad\text{and}\quad \mathcal O_{P_k(V)}(-1)|_{P_k(W)}=\mathcal O_{P_k(W)}(-1),
$$
and moreover the projection maps are of course compatible.

If $\varphi\colon (\mathbb C,0)\to X$ is a germ of holomorphic curve tangent to $V$ then we can define a projectivized lifting $\varphi_{[k]}\colon (\mathbb C,0)\to P_k(V)$. Such a lifting is tangent to $V_k$ and satisfies $\pi_{0,k}\circ\varphi_{[k]}=\varphi$.

It turns out that for any positive integer $m$ there is an isomorphism of sheaves
$$
(\pi_{0,k})_*\mathcal O_{P_k(V)}(m)\simeq\mathcal O_X(E_{k,m}V^*).
$$
By (the version in the invariant case \cite{Dem97} of) Theorem \ref{fund}, for any positive line bundle $A\to X$ and for all positive integers $k,m$, we have that 
$$
f_{[k]}(\mathbb C)\subset \operatorname{Bs}\bigl(\mathcal O_{P_k(V)}(m)\otimes\pi_{0,k}^*A^{-1}\bigr),
$$
whenever $f\colon\mathbb C\to X$ is an entire curve tangent to $V$ (here, $\operatorname{Bs}(\bullet)$ stands for base locus of a line bundle).

With this is mind, we define the \emph{Demailly--Semple locus $DS(X,V)$} to be
$$
DS(X,V)=DS_A(X,V),
$$
where
$$
DS_A(X,V)=\bigcap_{k\ge 1}\pi_{0,k}\left(\bigcap_{m\ge 1}\operatorname{Bs}\bigl(\mathcal O_{P_k(V)}(m)\otimes\pi_{0,k}^*A^{-1}\bigr)\setminus P_k(V)^{\textrm{sing}}\right).
$$
Here, $P_k(V)^{\textrm{sing}}$ is the complement of the set of values $\varphi_{[k]}(0)$ reached by all regular (\textsl{i.e.} with non--zero first derivative) germs of curves $\varphi$. 
The fact that $DS_A(X,V)$ is independent of the ample line bundle $A\to X$ can be shown in the same way as in Lemma \ref{ind}.

\begin{remark}
Observe that $P_1(V)^{\textrm{sing}}=\emptyset$, and $P_k(V)^{\textrm{sing}}\subset\operatorname{Bs}\bigl(\mathcal O_{P_k(V)}(m)\bigr)$ for all integer $m$ (see \cite{Dem97}). Moreover, $\pi_{0,k}(P_k(V)^{\textrm{sing}})=X$, for $k\ge 2$, but if $f\colon\mathbb C\to X$ is an entire curve tangent to $V$ such that $f_{[k]}(\mathbb C)\subset P_k(V)^{\textrm{sing}}$, then $f$ is in fact constant. Thus, it is necessary to remove $P_k(V)^{\textrm{sing}}$ from the base loci of the (anti)tautological bundles in order to get a really useful notion of Demailly--Semple locus (compare with the analogous definition in \cite[beginning of \S 13]{Dem97} and \cite[Introduction]{DT10}, where this minor point was not noticed).
\end{remark}

The relation
$$
\operatorname{Exc}(X,V)\subset GG(X,V)\subset DS(X,V)
$$
is immediate from definitions and Theorem \ref{fund}. \emph{We do not know wether $GG(X,V)=DS(X,V)$, for instance, under the natural hypothesis that $\det V^*$ is big.}

\begin{proof}[Proof of Theorem \ref{main} for Demailly--Semple jets]
Fix any ample line bundle $A\to X$. We shall construct for each integer $k>0$ a smooth manifold $X_k\subset P_k(V)$ which projects biholomorphically onto $X$ via $\pi_{0,k}$ and such that for all integers $m>0$ we have
$$
X_k\subset\operatorname{Bs}\bigl(\mathcal O_{P_k(V)}(m)\otimes\pi^*_{0,k}A^{-1}\bigr).
$$  
The starting datum is a directed manifold $(X,V)$ together with a rank one holomorphic subbundle $L\subset V\subset T_X$ such that $\kappa(L^{-1})<\dim X$. We define $X_k$ to be
$$
X_k:=P_k(L)\subset P_k(V).
$$
Since at each step we are always projectivizing a rank one vector bundle, all $X_k$ are isomorphic to the starting $X$, the isomorphism being given by the projections $\pi_{0,k}$. Moreover, since as we have seen
$$
\mathcal O_{P_k(V)}(-1)|_{P_k(L)}=\mathcal O_{P_k(L)}(-1)
$$
and, on the other hand, 
$$
\mathcal O_{P_k(L)}(-1)\simeq \pi_{0,k}^*L\simeq L,
$$
we deduce that the restriction of $\mathcal O_{P_k(V)}(m)\otimes\pi^*_{0,k}A^{-1}$ to $X_k=P_k(L)$ has no non--zero holomorphic sections for all positive integers $m$ by Kodaira's lemma, being isomorphic to $L^{-\otimes m}\otimes A^{-1}$. But then, it follows that the base locus of $\mathcal O_{P_k(V)}(m)\otimes\pi^*_{0,k}A^{-1}$ must necessarily contain $X_k$, and we are done. 
\end{proof}

We want to underline that we have indeed the much stronger property that not only $X_k$ is contained in the base locus but also that the restriction of $\mathcal O_{P_k(V)}(m)\otimes\pi^*_{0,k}A^{-1}$ to $X_k$ has itself no non--zero sections.

\section{Finite coverings and Green--Griffiths locus}

In this section we shall describe the behavior of the Green--Griffiths locus with respect to finite maps.
We begin with the following.

\begin{proposition}\label{finite}
Let $\rho\colon X'\to X$ be a finite surjective morphism of smooth projective manifolds and let $B\subsetneq X$ be its branch locus. Then, there is an inclusion $\rho\bigl(GG(X')\bigr)\subseteq GG(X)\cup B$. 

In particular, if $GG(X')=X'$, then $GG(X)=X$ and if $\rho$ is étale then $\rho\bigl(GG(X')\bigr)\subseteq GG(X)$.
\end{proposition}

\begin{proof}
Fix an ample line bundle $A\to X$ and suppose $x\in X\setminus\bigl(GG(X)\cup B\bigr)$. Since $x\not\in GG(X)$, there exists an integer $k_0>0$ such that for all non constant $k_0$-jet $\varphi_{k_0}\colon(\mathbb C,0)\to (X,x)$ there exists a positive integer $m$ and a global holomorphic section $P\in H^0(X,E_{k_0,m}^{GG}T^*_X\otimes A^{-1})$ such that $P(\varphi_{k_0}(0))\ne 0$. Next, let $y\in X'$ be such that $\rho(y)=x$. Since $x\not\in B$, then $y$ is not a ramification point and thus any non constant $k_0$-jet $\psi_{k_0}\colon(\mathbb C,0)\to (X',y)$ gives by composition with $\rho$ a non constant $k_0$-jet $\rho\circ\psi_{k_0}\colon(\mathbb C,0)\to (X,x)$ for which therefore there exists an integer $m>0$ and a global holomorphic section $P\in H^0(X,E_{k_0,m}^{GG}T^*_X\otimes A^{-1})$ such that $P\bigl((\rho\circ\psi_{k_0})(0)\bigr)\ne 0$. Thus, $\rho^*P$ is a global holomorphic section $\rho^*P\in H^0(X',E_{k_0,m}T^*_{X'}\otimes\rho^*A^{-1})$ such that $(\rho^*P)\bigl(\psi_{k_0}(0)\bigr)\ne 0$ (notice that $\rho^*A$ is ample since $\rho$ is a finite morphism), and therefore $y\not\in GG(X')$.

The last assertion is clear since both $GG(X)$ and $B$ are Zariski closed sets.
\end{proof}

\begin{corollary}\label{auto}
Let $\Phi \in \operatorname{Aut}(X)$ be an automorphism. Then, $\Phi(GG(X))=GG(X).$
\end{corollary}

\begin{remark}\label{relmorph}
Suppose that we have a surjective morphism between two projective manifolds $f\colon Y\to X$, and let $V\subset T_Y$ be a vector subbundle such that $df|_V\colon V\to f^*T_X$ is injective (we just mention the case of vector subbundles and not that of general subsheaves for the sake of simplicity). It is then immediate to check just following exactly the same proof of the proposition above that, more generally, we have $f\bigl(GG(Y,V)\bigr)\subset GG(X)$. This stronger version will be used in Section \ref{locirred}.
\end{remark}

For the converse, we suppose that the morphism is \'etale. We shall see later that this hypothesis is in fact necessary.

\begin{proposition}\label{invIm} 
If $\rho\colon X'\to X$ is a finite \'etale cover of smooth projective manifolds, then $\rho^{-1}\bigl(GG(X)\bigr)\subseteq GG(X')$.

In particular, if $GG(X)=X$, then $GG(X')=X'$.
\end{proposition}

\begin{proof}
Let us first reduce to the case when the finite \'etale cover is moreover Galois: this is done as follows. We claim that there exists a finite Galois \'etale cover $X''\to X$ which factorizes through $X'$. Once the proposition is known for Galois covers, just apply it to $X''\to X$ and then Proposition \ref{finite} to $X''\to X'$ to get the desired result as follows. 
Let $\mu\colon X''\to X' \to X$ be a Galois cover factorizing through $\nu\colon X'' \to X'$, $x \in GG(X)$ and $z\in\rho^{-1}(x)\subset X'$. Then, there exists $y\in\mu^{-1}(x)\subset X''$ such that $\nu(y)=z$. Since $\mu^{-1}(x) \subset GG(X'')$ then, by Proposition \ref{finite}, $\nu(y) \in GG(X')$ for any $y \in \mu^{-1}(x)$, and thus $z\in GG(X')$.

The claim is a consequence of the correspondence between subgroups of the fundamental group and \'etale covers and the following elementary lemma.

\begin{lemma}
Let $H\le \Gamma$ be a subgroup of finite index in $\Gamma$. Then, there exists a normal subgroup $N\trianglelefteq\Gamma$ of finite index in $\Gamma$ and such that $N$ is contained in $H$.
\end{lemma}

\begin{proof}
The group $\Gamma$ acts on the set of, say, left cosets $G/H$ by permutation. Then, we have a group homomorphism $\Gamma\to\mathfrak S(G/H)$ from $\Gamma$ to the finite group of permutation of $G/H$. It is straightforward to verify that its kernel $N$ has the required properties.
\end{proof}

So, let $\rho\colon X'\to X$ be a finite Galois \'etale cover, say of degree $d$, and let $\operatorname{Aut}(\rho)$ be its deck transformation group, which acts by biholomorphisms freely and transitively on each fiber, by hypothesis. Given an ample line bundle $A\to X$, and any global holomorphic section $Q\in H^0(X',E_{k,m}^{GG}T^*_{X'}\otimes\rho^*A^{-1})$, we can recover, by \lq\lq averaging\rq\rq{} $Q$ with the deck transformations, a global holomorphic section 
$$
\bigotimes_{\Phi\in\operatorname{Aut}(\rho)}\Phi^*Q=:\tilde P\in H^0\bigl(X',E_{k,dm}^{GG}T^*_{X'}\otimes\rho^*A^{-d}\bigr)
$$
invariant by $\operatorname{Aut}(\rho)$, which thus descends to a global holomorphic section $P\in H^0\bigl(X,E_{k,dm}^{GG}T^*_{X}\otimes A^{-d}\bigr)$.

If $y\in X'\setminus GG(X')$, take $k_0>0$ such that for all non constant $k_0$-jet $\psi_{k_0}\colon(\mathbb C,0)\to (X',y)$ there exists an integer $m>0$ and a global holomorphic section $Q\in H^0(X,E_{k_0,m}^{GG}T^*_{X'}\otimes \rho^*A^{-1})$ such that $Q(\psi_{k_0}(0))\ne 0$.

Let $\{y_1:=y,\dots,y_d\}$ be the orbit of $y$ under $\operatorname{Aut}(\rho)$. Then, by Corollary \ref{auto}, $y_i \in X'\setminus GG(X')$ for $i=1,\dots, d$. 

Next, consider a $k_0$-jet $\varphi_{k_0}\colon(\mathbb C,0)\to (X,x)$, where $x=\rho(y)$, and its (unique) lifting $\tilde\varphi^i_{k_0}$ to $X'$ such that $\tilde\varphi^i_{k_0}(0)=y_i$. Thus, we can find a $Q_i\in H^0(X,E_{k_0,m}^{GG}T^*_{X'}\otimes \rho^*A^{-1})$ such that $Q_i(\tilde\varphi^i_{k_0}(0))\ne 0$. Considering sections of the form $Q:=Q_1+\epsilon_2\,Q_2+\dots+\epsilon_d\,Q_d$ we see easily that choosing the scalars $\epsilon_i$ sufficiently small, we obtain a section $Q \in H^0(X,E_{k_0,m}^{GG}T^*_{X'}\otimes \rho^*A^{-1})$ such that $Q(\tilde\varphi^i_{k_0}(0))\ne 0$ for $i=1,\dots,d$. Indeed, by induction on $d$, suppose that $\epsilon_2,\dots,\epsilon_{d-1}$ are chosen in such a way that $S:=Q_1+\epsilon_2\,Q_2+\dots+\epsilon_{d-1}\,Q_{d-1}$ satisfies $S(\tilde\varphi^i_{k_0}(0))\ne 0$ for $i=1,\dots,d-1$. If $S(\tilde\varphi^d_{k_0}(0))\ne 0$ we are done taking $\epsilon_d=0$. Otherwise $S(\tilde\varphi^d_{k_0}(0))= 0$ and we choose $\epsilon_d \neq 0$ sufficiently small so that $S(\tilde\varphi^i_{k_0}(0))+\epsilon_d\,Q_d(\tilde\varphi^i_{k_0}(0))\ne 0$ for $i=1,\dots,d-1$.

Finally, the averaging process above applied to this particular $Q$ then gives a $P\in H^0\bigl(X,E_{k_0,dm}^{GG}T^*_{X}\otimes A^{-d}\bigr)$ such that $P(\varphi_{k_0}(0))\ne 0$, as it is straightforwardly checked, and $x\not\in GG(X)$.
\end{proof}

\begin{remark}\label{rmkGalois}
Moreover, the averaging process above shows that in the Galois situation if $H^0(X',E_{k,m}^{GG}T^*_{X'})\ne 0$ for some $k,m>0$, then $H^0(X,E_{k,dm}^{GG}T^*_{X})\ne 0$ where $d$ is the degree of the covering.
\end{remark}

The next example, together with the remark above, shows that the \'etale assumption is necessary.  
 
\begin{example}
Let $X\subset\mathbb P^{n}$ a smooth hypersurface of degree $d$, cut out by the single equation $P(z_0,\dots,z_n)=0$. Consider the smooth hypersurface $X'\subset\mathbb P^{n+1}$ cut out by the equation $z_{n+1}^d-P(z_0,\dots,z_n)=0$. The map $\pi\colon X'\to\mathbb P^{n}$ given by the projection onto the first $n+1$ coordinates is a branched Galois cover whose automorphism group is $\mathbb Z/d\mathbb Z$ (acting by multiplying $z_{n+1}$ by powers of a primitive $d$th root of unity), ramified along $X$. By \cite{Div09}, we know that if $d$ is large enough, then $H^0(X',E_{k,m}^{GG}T^*_{X'})\ne 0$ for $m\gg k\gg 1$. On the other hand, since $T_{\mathbb P^{n}}$ is ample, for all $k$-tuple of integers $\ell_1,\dots,\ell_k>0$, the graded terms $H^0(\mathbb P^{n},S^{\ell_1}T^*_{\mathbb P^n}\otimes\cdots\otimes S^{\ell_k}T^*_{\mathbb P^n})=0$ and thus $H^0(\mathbb P^n,E^{GG}_{k,m}T^*_{\mathbb P^n})=0$ for all $k,m>0$. 

Let us finally remark that one knows from \cite{DMR10} that generic projective hypersurfaces of high degree $X \subset \mathbb P^{n}$ have $GG(X) \neq X$.
\end{example}
\section{Green--Griffiths locus for product manifolds}\label{sectionproduct}

In this section we will show the following consequence of Theorem \ref{main} (compare with Corollary \ref{introproduct}).

\begin{corollary}\label{product}
Let $X$ be a complex projective manifold and suppose that $X$ splits, up to finite (possibly branched) coverings, into a nontrivial product $Y\times Z$. Then, $GG(X)=X$.
\end{corollary} 

To prove Corollary \ref{product}, we will need the following two elementary lemmata.

\begin{lemma}\label{growth}
Let $E\to X$ be a rank $r$ holomorphic vector bundle on a projective manifold $X$ of dimension $n$, and let $A\to X$ be an ample line bundle. If, for some integer $m_0>0$, $H^0(X,S^{m_0}E\otimes A^{-1})\ne 0$, then there exists a constant $C>0$ such that for all sufficiently divisible integers $m>0$ we have
$$
C^{-1}\,m^{n+r-1}\le h^0(X,S^mE)\le C\, m^{n+r-1}.
$$
Moreover, this is the maximum possible growth rate for the dimension of the spaces of global sections of symmetric powers of $E$.
\end{lemma}

\begin{proof}
Let $\pi\colon\mathbb P(E)\to X$ the projective bundle of hyperplanes of $E$ and $\mathcal O_E(1)\to\mathbb P(E)$ the tautological quotient line bundle associated to $E$. We have that $\dim \mathbb P(E)=n+r-1$ and, using the projection formula, it is well known that 
\begin{equation}\label{projectionformula}
H^0\bigl(\mathbb P(E),\mathcal O_E(m)\otimes\pi^*G\bigr)\simeq H^0\bigl(X,S^mE\otimes G\bigr),
\end{equation}
for any line bundle $G\to X$. Since $\mathcal O_E(1)$ is $\pi$-ample, there exists a positive integer $\ell_0$ such that $\mathcal O_E(1)\otimes\pi^*A^{\ell_0}$ is ample. On the other hand, by hypothesis and (\ref{projectionformula}), we have that $\mathcal O_E(m_0)\otimes\pi^*A^{-1}$ is effective. Now, fix any ample line bundle $F\to\mathbb P(E)$ and let $k_0$ be a positive integer such that $\mathcal O_E(k_0)\otimes\pi^*A^{\ell_0k_0}\otimes F^{-1}$ is effective. Then,
$$
\mathcal O_E\bigl(k_0(m_0\ell_0+1)\bigr)\otimes F^{-1}\simeq\mathcal O_E(m_0\ell_0k_0)\otimes\pi^*A^{-\ell_0k_0}\otimes\mathcal O_E(k_0)\otimes\pi^*A^{\ell_0k_0}\otimes F^{-1}
$$
is effective.

Thus, by Kodaira's lemma, $\mathcal O_E(1)$ is big and there exists a constant $C>0$ such that for all sufficient divisible integers $m>0$ we have
$$
C^{-1}\,m^{n+r-1}\le h^0\bigl(\mathbb P(E),\mathcal O_E(m)\bigr)\le C\, m^{n+r-1},
$$
and this is the maximum possible growth rate of the dimension of the spaces of global sections of powers of $\mathcal O_E(1)$. Thanks to (\ref{projectionformula}), the same holds for symmetric powers of $E$ on $X$.
\end{proof}

\begin{remark}
It is a standard fact that for \lq\lq sufficiently divisible integer $m$\rq\rq{} it is enough to read \lq\lq such that $m$ is large enough and $S^mE$ has at least a non zero global section\rq\rq{}.
\end{remark}

\begin{lemma}\label{pullback}
Let $f\colon X\to Y$ be a surjective morphism with connected fibers of smooth projective manifolds and suppose that $n=\dim X>\dim Y=m$. If $E\to Y$ is any holomorphic vector bundle, then for any ample line bundle $A\to X$ we have $H^0(X,f^*E\otimes A^{-1})=0$. 
\end{lemma}

\begin{proof}
Suppose that $E$ is of rank $r$. If, by contradiction, $H^0(X,f^*E\otimes A^{-1})\ne0$, then, by Lemma \ref{growth}, there exists a positive constant $C$ such that
$$
C^{-1}\,k^{n+r-1}\le h^0(X,f^*S^kE)\le C\, k^{n+r-1},
$$
for all sufficiently divisible integer $k$. On the other hand, by the projection formula, $H^0(X,f^*V)\simeq H^0(Y,V)$, for any holomorphic vector bundle $V\to Y$. But then, 
$$
C^{-1}\,k^{n+r-1}\le h^0(X,f^*S^kE)=h^0(Y,S^kE)\le B\,k^{m+r-1},
$$
for some positive constant $B$ and all sufficiently divisible $k$, contradiction.
\end{proof}

We are now in a good shape to give a short proof of Corollary \ref{product}.

\begin{proof}[Proof of Corollary \ref{product}]
Let $\rho\colon\tilde X\to X$ be a finite surjective morphism such that $\tilde X\simeq Y\times Z$ and let $A\to\tilde X$ be an ample line bundle. Then, $0<\dim Y<\dim\tilde X$ and $\operatorname{pr}_1^*T_Y$ is a holomorphic subbundle of $T_{\tilde X}$, where $\operatorname{pr}_1\colon\tilde X\to Y$ is the first projection. For any integer $m>0$, by Lemma \ref{pullback} applied to $(T^*_Y)^{\otimes m}$, we have that $H^0\bigl(\tilde X,(\operatorname{pr}_1^*T^*_Y)^{\otimes m}\otimes A^{-1}\bigr)=0$, so that by Theorem \ref{main} we have $GG(\tilde X)=\tilde X$. Therefore, by Proposition \ref{finite}, $GG(X)=X$.
\end{proof}

\subsection*{Examples}
Using the preceding results it is already possible to produce many examples of projective manifolds of general type covered by their Green--Griffiths locus. Let us consider the case of surfaces, although easy generalizations to higher dimensions can be given.

The easiest applications of Corollary \ref{product} is certainly to take $X=C_1\times C_2$ a product of $2$ compact Riemann surfaces of genus $g_i \geq 2$ which gives an example of general type and even hyperbolic, as observed in the Introduction.

Starting from these examples and using Proposition \ref{finite}, one can construct more interesting examples taking quotient by finite groups. 

First, let us consider symmetric product of curves. Given a smooth projective curve $C$ of genus $g$, take $X:=C^{(2)}=C\times C/\mathfrak{S}_2,$ the quotient by the symmetric group. If $g\geq 3$ then $C^{(2)}$ is of general type, hyperbolic \cite{SZ00} and satisfies $GG(X)=X.$ 

One can also consider the algebraic surfaces whose canonical models arise as quotients $X = (C_1\times C_2)/G$ of the product $C_1\times C_2$ of two curves of genera $g_1:=g(C_1),g_2:=g(C_2)\ge 2$, by the action of a finite group $G$ (in other words, $X$ has only rational double points as singularities).
The minimal resolution $S$ of the quotient $X=(C_1\times C_2)/G$ is called a product--quotient surface and has been intensively studied in \cite{BCGP12}. This gives examples of surfaces of general type $S$ such that $GG(S)=S$ which are not hyperbolic whenever the action of $G$ is not free (\textsl{i.e.} whenever $X$ is singular) and can have in some cases finite fundamental group. 

In particular considering the universal cover, one obtains simply connected examples of surfaces of general type covered by their Green--Griffiths locus. Moreover, unlike previous examples, the distribution of rational, elliptic or entire curves on these surfaces seems not to be known.

\section{Green--Griffiths locus of locally reducible quotients of bounded symmetric domains}

For basic results and notation about hermitian symmetric domains, we refer to \cite{Mok89}. Let $\Omega=G_0/K$ be a bounded symmetric domain, where $G_0=\operatorname{Aut}_0(\Omega)$ is the connected component of the identity of the automorphisms group of $\Omega$, and let $\Gamma\subset\operatorname{Aut}(\Omega)$ be a lattice. The quotient $X=\Omega/\Gamma$ is said to be \emph{locally irreducible} (resp. \emph{locally reducible}) if $\Omega$ is irreducible (resp. reducible) as a hermitian symmetric space.

\begin{definition}
We say that $X=\Omega/\Gamma$ is \emph{reducible} if there exists a subgroup $\Gamma_0\subset\Gamma$ of finite index and a decomposition $\Omega\simeq\Omega_1\times\Omega_2$ into a product of bounded symmetric domains $\Omega_1$ and $\Omega_2$ such that $\Gamma_0\simeq\Gamma_1\times\Gamma_2$ and $\Gamma_i\subset\operatorname{Aut}(\Omega_i)$, $i=1,2$. Otherwise, $X$ is said to be \emph{irreducible}. 
\end{definition} 

In particular, a reducible quotient is by definition locally reducible. Geometrically, the reducibility of $X$ means that, up to a finite covering, $X$ can be decomposed isometrically in a non trivial way. We now concentrate on the locally reducible case. 

\begin{theorem}
Let $X$ be a complex projective manifold uniformized by a reducible bounded symmetric domain. Then, $GG(X)=X$.
\end{theorem}

\begin{proof} 
The proof splits naturally into two parts: the reducible and the irreducible case. We first treat the reducible case.

\subsubsection*{The case of reducible quotients}

Let $X$ be a reducible quotient. By definition, there exists a decomposition $\Omega\simeq\Omega_1\times\Omega_2$ into a product of bounded symmetric domains and a finite \'etale covering $\tilde X$ of $X$ which decomposes into a product $\tilde X\simeq Y\times Z$ whose factors are uniformized respectively by $\Omega'$ and $\Omega''$. Since $\tilde X$ is a product, then, by Corollary \ref{product}, $GG(X)=X$.

\subsubsection*{The case of irreducible quotients}

Let us treat first the particular case of an irreducible quotient of a polydisk, in order to give the flavor of the proof in a very explicit context.

Fix a co--compact subgroup $\Gamma\subset \operatorname{PSL}(2,\mathbb R)^n$
acting freely and properly discontinuously on $\Delta^n$. Suppose that $\Delta^n/\Gamma$ is irreducible, so that by the so--called Density Lemma \cite[Cor.(5.21) and Thm.(5.22), p.86]{Rag72} the projection
$$
\Gamma\to\operatorname{PSL}(2,\mathbb R)
$$
onto, say, the first factor has dense image $\Gamma_1$ (it is indeed dense onto every factor). Call $z\in\Delta$ the complex coordinate of the first factor and $w=(w_1,\dots,w_{n-1})\in\Delta^{n-1}$ the complex coordinates of the last $n-1$ factors. We begin with a preliminary elementary lemma.

\begin{lemma}[Compare with \cite{She95}]\label{sb}
Let $\eta=f(z,w)\,(dz)^{\otimes m}$ be a symmetric differential of degree $m$ on $\Delta_z\times\Delta_w^{n-1}$. Suppose that $\eta$ is $\Gamma$-invariant. Then, $\eta$ vanishes identically. 
\end{lemma}

\begin{remark}\label{automorphic}
This lemma should be regarded as a consequence of the following classical result for automorphic forms. Let $\Gamma\subset \operatorname{PSL}(2,\mathbb R)^n$ be an irreducible discrete subgroup with compact quotient $\Delta^n/\Gamma$, and $f$ a $\Gamma$-automorphic form of weight $2r$, $r=(r_1,\dots,r_n)\in\mathbb Z^n$. If $r_1\cdots r_n=0$, then $f$ must vanish identically. The same statement holds true for instance in the following more general setting: the quotient $\Delta^n/\Gamma$ is not necessarily compact but $\Gamma$ is commensurable with the Hilbert modular group (for these and related statements see for instance \cite{Fre90}). This latter version will be used in Subsection \ref{siegel}.
\end{remark}

We shall come back about the existence of such groups $\Gamma$ in Subsection \ref{Gamma}, and now we give an elementary proof of the above lemma.

\begin{proof}
Consider the smooth real function given by taking the Poincaré norm of $\eta$:
$$
(z,w)\mapsto |f(z,w)|(1-|z|^2)^{m}.
$$
This is a $\Gamma$-invariant smooth function, thus defined on the compact quotient $X=\Delta^n/\Gamma$. Let $p\in X$ be a point where this function attains its maximum and consider the discrete set of points $\{(z_{(j)},w_{(j)})\}_{j\in J}\in\Delta\times\Delta^{n-1}$ which are in the preimage of $p$ by the quotient map. For each $j\in J$, the holomorphic map defined on the polydisk $\{z_{(j)}\}\times\Delta_w$ by
$$
w\mapsto f(z_{(j)},w)
$$
attains a maximum at the interior point $w_{(j)}$ and so it is constant. So, all the $w_\lambda$-derivatives vanish:
$$
\partial f/\partial w_\lambda (z_{(j)},w)=0,\quad\text{for all $j\in J$, $\lambda=1,\dots,n-1$, and $w\in\Delta^{n-1}$}.
$$ 
By the density of $\Gamma_1$, the set $\{z_{(j)}\}_{j\in J}$ is dense in $\Delta$, so $\partial f/\partial w_\lambda\equiv 0$, $\lambda=1,\dots,n-1$, and $f$ does not depend on $w$. Therefore $\eta$ does only depend on $z$ and can thus be regarded as a symmetric differential of degree $m$ on $\Delta_z$, invariant by the action of $\Gamma_1$ on $\operatorname{Aut}(\Delta_z)$ which has, once again by hypothesis, dense image. But then, $\eta\equiv 0$. 
\end{proof}

Now, consider the projective $n$-dimensional complex manifold $X=\Delta^n/\Gamma$.
The holomorphic foliation by disks $\widetilde{\mathcal F}$ on $\Delta^n$ generated by $\partial/\partial z$ descends to a smooth foliation by curves $\mathcal F$ on $X$. Consider its tangent bundle $T_{\mathcal F}$: it is a rank one holomorphic subbundle of the tangent bundle $T_X$ of $X$.

\begin{proposition}[Compare for instance with \cite{Bru04}]\label{modularfoliation}
The canonical bundle $K_\mathcal F=T_\mathcal F^*$ of $\mathcal F$ has negative Kodaira--Iitaka dimension.
\end{proposition}

\begin{proof}
We have to show that $H^0(X,K_\mathcal F^{\otimes \ell})=\{0\}$ for all integers $\ell>0$. The latter space of global sections identifies canonically with the space of $\Gamma$-invariant global holomorphic sections of $K_{\widetilde{\mathcal F}}^{\otimes\ell}$ over $\Delta^n$. Since $\widetilde{\mathcal F}$ is generated by $\partial/\partial z$, these sections are exactly of the form considered in Lemma \ref{sb}, and therefore they vanish identically. 
\end{proof}

In particular $K_\mathcal F$ is not big and, by Kodaira's lemma, for any ample line bundle $A\to X$ and for any integer $m>0$, we have that $H^0(X,K_\mathcal F^{\otimes m}\otimes A^{-1})=0$, so that Theorem \ref{main} applies.

\smallskip

Now, we pass to the general case. Let $\Omega=\Omega_1\times\cdots\times\Omega_k$, $k\ge 2$, be the decomposition of $\Omega$ into irreducible components. A classical theorem of Cartan (which can be found in \cite[Chap. 5]{Nar71}) states that all automorphisms of $\Omega$ are given by automorphisms of individual irreducible factors and by permutation of isomorphic factors. It follows that for any lattice $\Gamma\subset\operatorname{Aut}(\Omega)$, there is a subgroup $\Gamma_0\subset\Gamma$ of finite index such that $\Gamma_0\subset\operatorname{Aut}_0(\Omega)=\operatorname{Aut}(\Omega_1)\times\cdots\times\operatorname{Aut}(\Omega_k)$. Thanks to Proposition \ref{finite}, we can thus suppose without loss of generality that the fundamental group of $X$, seen as a lattice inside $\operatorname{Aut}(\Omega)$, is contained in the connected component of the identity $\operatorname{Aut}_0(\Omega)$. In particular we can suppose that the action of $\Gamma$ preserves the factors and thus we have a corresponding splitting 
$$
T_X\simeq V_1\oplus\cdots\oplus V_k.
$$
of the tangent bundle of $X$ such that $\pi^* V_i=T_{\Omega_i}$, where $\pi\colon\Omega\to X$ is the quotient map.

Let $K_X$ be the canonical bundle of $X$. Since it is ample there exists an integer $\ell>0$ such that $K_X^{\ell}$ is effective. We shall show that for any $i=1,\dots k$, and for any integer $m>0$, we have that $H^0\bigl(X,(V_i^*)^{\otimes m}\otimes K_X^{-\ell}\bigr)=0$, so that $GG(X)=X$ by Theorem \ref{main} (and in fact, more generally, that $GG(X,V)=X$, for any subbundle $V\subset T_X$ containing one of the $V_i$'s). For $i=1,\dots,k$, let
$$
\Omega_i=G_i/K_i,\quad G_i=\operatorname{Aut}_0(\Omega_i)\quad\text{and}\quad\mathfrak g_i=\mathfrak k_i+\mathfrak m_i,
$$
where $\mathfrak g_i$ and $\mathfrak k_i$ are respectively the Lie algebras of $G_i$ and $K_i$ and $\mathfrak m_i\simeq \mathfrak g_i/\mathfrak k_i=T_{\Omega_i,eK_i}$ (here $eK_i$ is the identity coset).
Next, let 
$$
\varrho_i\colon K_i\to\operatorname{GL}(\mathfrak m_i)
$$ 
be the isotropy representation of $K_i$ on $\mathfrak m_i$, so that the corresponding homogeneous vector bundle is the tangent bundle $T_{\Omega_i}$. The irreducible locally homogeneous vector bundle $V_i\to X=\Omega/\Gamma$ is thus the one associated to the irreducible representation
$$
\begin{aligned}
\sigma_i\colon & K_1\times\dots\times K_i\times\dots\times K_k\to\operatorname{GL}(\mathfrak m_i) \\
& (g_1,\dots,g_i,\dots,g_k)\mapsto\varrho_i(h_i).
\end{aligned}
$$
Now, fix any integer $m>0$ and let $(\sigma_i^*)^{\otimes m}$ be the representation defining $(V_i^*)^{\otimes m}$. Since $K_i$ is compact, all its representations completely decompose into direct sum of irreducible ones, thus:
$$
(\mathfrak m_i^*)^{\otimes m}\simeq E_{1}\oplus\cdots\oplus E_{N}
$$
and we have a corresponding decomposition of $(V_i^*)^{\otimes m}$ into a direct sum of irreducible locally homogeneous vector bundles
$$
(V_i^*)^{\otimes m}\simeq W_{1}\oplus\cdots\oplus W_{N}. 
$$
Now, they all have a natural induced hermitian metric coming from the Bergman metric on $\Omega_i$, and the decomposition is as hermitian vector bundles. Moreover, observe that each of the $W_j$'s comes from a homogeneous vector bundle on $\Omega$ which is in fact a pull-back of a homogeneous vector bundle on $\Omega_i$. Therefore, the Chern curvature of the $W_j$'s is zero whenever evaluated (in its $(1,1)$-form part) on tangent vectors lying on the orthogonal complement of $V_i$. In particular, the $W_j$'s cannot be of strictly positive Chern curvature in the sense of Griffiths. 

To conclude, we use the following deep vanishing theorem which we rephrase slightly. 

\begin{theorem}[See {\cite[Corollary 1' on page 212]{Mok89}}]\label{mokvanishing}
Let $\Omega$ be a bounded symmetric domain of complex dimension $\ge 2$ and $X=\Omega/\Gamma$ be an irreducible quotient of finite volume of $\Omega$ by a torsion free discrete group $\Gamma$ of automorphisms. Suppose $V$ is an irreducible locally homogeneous Hermitian vector bundle on $X$ which is not of strictly positive Chern curvature in the sense of Griffiths. Then, $H^0(X,V)= 0$ unless $V$ is trivial.
\end{theorem}

Now, if $W_j$ is non trivial then, by Theorem \ref{mokvanishing}, $H^0(X,W_j)=0$ and \textsl{a fortiori} $H^0(X,W_j\otimes K_X^{-\ell})=0$ since $K_X^\ell$ is effective. If, on the other hand, $W_j$ is trivial, then 
$$
W_j\otimes K_X^{-\ell}\simeq\bigl(K_X^{-\ell}\bigr)^{\oplus\operatorname{rk}W_j}
$$
and thus $H^0(X,W_j\otimes K_X^{-\ell})\simeq H^0\left(X,\bigl(K_X^{-\ell}\bigr)^{\oplus\operatorname{rk}W_j}\right)=0$, since $K_X$ is ample.
\end{proof}

\begin{remark}
Although we were not able to prove it, we strongly expect that in the proof above the case $W_j$ trivial should not exist. This would give the stronger statement that for any $i=1,\dots k$, and for any integer $m>0$,
$$
H^0\bigl(X,(V_i^*)^{\otimes m}\bigr)=0.
$$
\end{remark}

\subsection{The case of surfaces}

Let us now speculate more on the case of surfaces. In this case $\Omega$ is necessarily the bidisk $\Delta^2$. We begin with a few words about the existence of irreducible co--compact subgroups $\Gamma\subset\operatorname{PSL}(2,\mathbb R)^2$ acting freely and properly discontinuously on $\Delta^2$, which is however quite classical (see for instance in \cite{Sha78}).

\subsubsection{Construction of $\,\Gamma$}\label{Gamma}

Take a quaternion algebra $A$ which is division and whose center is a totally real quadratic number field $k$ (for an excellent reference about quaternion algebras and Fuchsian groups see \cite{Kat92}). Assume that
$$
A\otimes_{\mathbb Q}\mathbb R=M(2,\mathbb R)^2,
$$
that is, $A$ is unramified at the two places corresponding to the two different embeddings of $k$ into $\mathbb R$.
If $\mathfrak D$ is a maximal order in $A$, denote by $\Gamma(1)$ the group of units in $\mathfrak D$ with reduced norm $1$, and identify it with its isomorphic image in $\operatorname{SL}(2,\mathbb R)^2$. Next, call $\Gamma'=\Gamma(1)/\{\pm 1\}$ the image of $\Gamma(1)$ in $\operatorname{PSL}(2,\mathbb R)^2$. Then, it is well known that the action of $\Gamma'$ on $\Delta^2$ is irreducible and properly discontinuous (see \cite{Sha78,Shi94}); moreover, the fact that $A$ is division classically implies that this action is co--compact. 

Now, since the action is co--compact, $\Gamma'$ is a group of finite type. By Selberg's theorem, a finitely generated linear group over a field of zero characteristic is virtually torsion--free, \textsl{i.e.} it has some finite index subgroup $\Gamma$ which is torsion--free. Thus, $\Gamma$ acts freely on $\Delta^2$. Moreover, being of finite index in $\Gamma'$, it is straightforward to see that it is again irreducible and with a properly discontinuous, co--compact action. 

\subsubsection{Surfaces of general type with a holomorphic foliation by curves}

Let $S$ be a surface of general type and $L\subset T_S$ a holomorphic line subbundle, \textsl{i.e.} $S$ is endowed with a smooth holomorphic foliation $\mathcal F$ by curves whose tangent bundle $T_\mathcal F$ is $L$. Suppose that the canonical bundle $K_\mathcal F$ is not big.

\begin{proposition}\label{smooth}
Let $(S,\mathcal F)$ be as above. Then, $S$ is a quotient of the bidisk $\Delta^2$.
\end{proposition}

\begin{proof}
Following \cite{Bru97}, the existence of such a smooth foliation on the surface of general type $S$ implies that $K_S$ is ample and so, by Aubin--Yau, $S$ is Kähler--Einstein and hence $T_S$ is $K_S$-semistable. The semistability inequality reads
\begin{equation}\label{semistability}
c_1(T_\mathcal F)\cdot c_1(S)\ge \frac 12\,c_1(S)^2>0.
\end{equation}
Since $\mathcal F$ is smooth, the Baum--Bott formulae give
\begin{equation}\label{BB}
\begin{aligned}
& c_2(S)-c_1(T_\mathcal F)\cdot c_1(S)+c_1(T_\mathcal F)^2=0 \\
& c_1(S)^2-2\,c_1(T_\mathcal F)\cdot c_1(S)+c_1(T_\mathcal F)^2=0.
\end{aligned}
\end{equation}
If $T_S$ is stable, the first inequality in (\ref{semistability}) is strict and using the second of (\ref{BB}) we obtain $c_1(T_\mathcal F)^2>0$. Thus, $T_\mathcal F$ or its dual $K_\mathcal F$ must be big. Since, by (\ref{semistability}), $c_1(K_\mathcal F)\cdot c_1(K_S)>0$, we get that $K_\mathcal F$ is big, contradiction.

Therefore, $T_S$ is polystable but not stable, that is $T_S=L_1\oplus L_2$ is a direct sum of two line bundles and by Beauville--Yau's uniformization theorem \cite{Bea00,Yau93}, the universal cover $\widetilde S$ of $S$ splits as a product of simply connected Riemann surfaces, the decomposition of the tangent bundle lifts and the fundamental group of $S$ acts diagonally on $\widetilde S$. Of course, the only possibility for $\widetilde S$ is to be the product of two disks.
\end{proof}

It is known by \cite{Lu96} that the surfaces as above satisfy the Green--Griffiths conjecture since they verify the Chern numbers inequality $c_1(S)^2-2\,c_2(S)\ge 0$ (to see this just inject the difference of the two identities in (\ref{BB}) into the first inequality in (\ref{semistability})).

Next, assume that $S$ is a smooth projective surface of general type and $\mathcal F$ a (possibly singular, with at most isolated singularities) holomorphic foliation by curves whose canonical bundle $K_\mathcal F$ is not big. By Seidenberg's theorem, we can suppose without loss of generality that $\mathcal F$ has reduced (or canonical) singularities.

Thus, the birational classification of foliations developed by Brunella and McQuillan (see \cite{Bru97,McQ08} and \cite{Bru04}) tells us that $\mathcal F$ is necessarily of the following two types.
\begin{itemize}
\item A Hilbert modular foliation, and thus $S$ is a Hilbert modular surface, if $\kappa(K_\mathcal F)=-\infty$.
\item An isotrivial fibration of genus $\ge 2$, if $\kappa(K_\mathcal F)=1$.
\end{itemize}

This gives the \lq\lq singular\rq\rq{} analogous of the Proposition \ref{smooth}. Then, for instance, Hilbert modular surfaces which are minimal resolution of surfaces with cusps give examples of surfaces $S$ with $c_1(S)^2 < 2\,c_2(S)$ such that $GG(S) = S$. 

Finally, what if $K_\mathcal F$ is big? Here is a natural question which was told to us by M. McQuillan in a private communication.

\begin{question}
Let $S$ be a surface admitting a holomorphic foliation by curves $\mathcal F$ with canonical singularities. Suppose that $K_\mathcal F$ is big. Is it then true that $GG(S,T_\mathcal F) \subsetneq S$? If moreover $S$ is of general type, what can be said about $GG(S)$ in this case?
\end{question}

\section{Green--Griffiths locus of locally irreducible quotients of bounded symmetric domains}\label{locirred}

In this section, we will prove the following theorem.

\begin{theorem}\label{locirredthm}
Let $X$ be a complex projective manifold uniformized by an irreducible bounded symmetric domain of rank greater than or equal to two. Then, $GG(X)=X$.
\end{theorem}

Let us explain the idea to prove such a result. Recall that in the locally reducible case we used in an essential way the existence of natural holomorphic foliations on the manifold coming from the irreducible factors of the universal cover. At a first glance, it could be therefore tempting to think that the irreducibility of the tangent bundle could be somehow an obstruction for the Green--Griffiths locus to cover the whole manifold. The theorem above tells us that this is indeed not the case.

To prove Theorem \ref{locirredthm} we need thus to find something which replaces the existence of natural foliations. The idea is somehow to use, instead of the trivial foliations on the universal cover $\Omega$ of $X$, the existence of many polydisks nicely embedded in $\Omega$.

\subsection{The Green--Griffiths locus of Shimura varieties: an \lq\lq all or nothing\rq\rq{} principle \`a la Ullmo--Yafaev.}\label{UY}
It is not easy in general to construct totally geodesic subvarieties in arbitrary compact quotients of symmetric hermitian domains. Interesting examples are given by the theory of Shimura varieties. Consider a connected semisimple and simply connected $\mathbb Q$-anisotropic linear algebraic group $\mathbb G$ over $\mathbb Q$ with associated Lie group $G=\mathbb G(\mathbb R).$ Then, for any congruence subgroup $\Gamma$ which acts without fixed point on $D=G/K$, the quotient space $X=D/\Gamma$ is a smooth projective variety.

Any $g\in \mathbb G_{\mathbb Q}$ gives rise to a Hecke correspondence $T_g$ on $X$ as follows. Let $\Gamma_g=\Gamma \cap g^{-1}\Gamma g$: this is a congruence subgroup of finite index in $\Gamma$. The variety $Y=D/\Gamma_g$ is smooth projective and admits two finite \'etale maps to $X$, given by
$p_1(\Gamma_gx)=\Gamma x$, $p_2(\Gamma_g x)=\Gamma gx,$ for any $x \in D$. The action on cycles of $X$ is given by ${p_2}_*p_1^*.$

A direct application of Propositions \ref{finite} and \ref{invIm} gives the following result which should be seen as a geometric counterpart of the \lq\lq all or nothing\rq\rq{} principle of \cite{UY10} and as a preliminary version of Theorem \ref{locirredthm}.

\begin{theorem}\label{AON}
Let $X$ be a compact Shimura variety associated to a semisimple, connected and simply connected $\mathbb Q$-anisotropic algebraic group over $\mathbb Q$. Then $GG(X)=\emptyset$ or $GG(X)=X$.
\end{theorem}
\begin{proof}
Since the group $G$ is connected, $\mathbb G_{\mathbb Q}$ is dense in $G=\mathbb G(\mathbb R)$ \cite{PR94}. It follows that if $Y \subset X$ is a subvariety of $X$, the irreducible components of $T_g(Y)$ for $g \in \mathbb G_{\mathbb Q}$ are dense in $X$.

We conclude using Propositions \ref{finite} and \ref{invIm} which give that if $Y \subset GG(X)$ then $T_g(Y) \subset GG(X).$
\end{proof}

Now, we give a criterion to decide on which side of the alternative we are.

\begin{corollary}
Let $X$ be a compact Shimura variety associated to a semisimple, connected and simply connected $\mathbb Q$-anisotropic algebraic group over $\mathbb Q$. Suppose that $X$ contains a subvariety $Y$ of positive dimension such that $GG(Y)=Y$. Then, $GG(X)=X$.
\end{corollary}

\begin{proof}
Fix an ample line bundle $A\to X$. First we observe that if $GG(Y)=Y$ then $Y\subset GG(X)$. Indeed, if $y\in Y$, for all integer $k>0$ there exists a $k$-jet $\varphi_{k}\colon(\mathbb C,0)\to (Y,y)$ such that for all integer $m>0$ and all $Q \in H^0(Y,E_{k,m}^{GG}T^*_{Y}\otimes A^{-1}|_Y)$, one has $Q(\varphi_{k}(0))= 0$. Given any $P \in H^0(X,E_{k,m}^{GG}T^*_{X}\otimes A^{-1})$, restricting $P$ to $Y$ we obtain $P(\varphi_{k}(0))= 0$. But then, $y \in GG(X)$.

Finally, since then $GG(X)\ne\emptyset$, thanks to Theorem \ref{AON} we obtain $GG(X)=X.$
\end{proof}

Although we have considered the compact case, the same strategy can be applied in the isotropic case. 
An interesting and very explicit case is the one of Siegel modular varieties which we shall discuss now in some details.

\subsection{The Green--Griffiths locus of (compactifications of) Siegel modular varieties}\label{siegel}
Let $n\ge 2$ and 
$$
\Omega=D^{III}_n:=\{Z\in M(n,\mathbb C)\mid Z={^\dag Z},\quad\operatorname{Id}_n-{^\dag\overline Z}\cdot Z>0\}
$$ 
be the classical bounded symmetric domain of type $III$, which is holomorphically equivalent to the be the Siegel upper half--space $\mathbb H_n=\{\tau\in M(n,\mathbb C)\mid\tau={^\dag}\tau,\Im\tau>0\}$. It has complex dimension $n(n+1)/2$ and rank $n$. The group $\operatorname{Sp}(2n,\mathbb R)$ acts transitively on $\Omega$ and we have indeed a presentation 
$$
\Omega=\operatorname{Sp}(2n,\mathbb R)/U(n)
$$
as a homogeneous space. Now, let $\Gamma\subset\operatorname{Sp}(2n,\mathbb R)$ be a lattice commensurable with $\operatorname{Sp}(2n,\mathbb Z)$, and consider the quotient manifold $X=\Omega/\Gamma$ and any smooth compactification $\overline X$ of $X$.

\begin{proposition}
The Green--Griffiths locus $GG\bigl(\overline{X}\bigr)$ of $\overline X$ is the whole manifold.
\end{proposition}

\begin{proof}
There is a totally geodesic polydisk $\Delta^n \hookrightarrow \Omega$ given by 
$$
\Delta^n\ni z=(z_1,\dots, z_n) \mapsto z^*=\operatorname{diag}(z_1,\dots, z_n),
$$
where $\operatorname{diag}(z_1,\dots, z_n)$ is the diagonal matrix with $(z_1,\dots, z_n)$ as entries along the diagonal. This corresponds to the embedding 
$$
\begin{aligned}
& \operatorname{SL}(2, \mathbb R)^n \hookrightarrow\operatorname{Sp}(2n,\mathbb R) \\
& M=(M_1,\dots,M_n) \mapsto M^*=
\begin{pmatrix}
a^* & b^* \\
c^* & d^*
\end{pmatrix}
\end{aligned}
$$ 
where, for $i=1,\dots,n$, 
$$
M_i=
\begin{pmatrix}
a_i & b_i \\
c_i & d_i
\end{pmatrix}\in\operatorname{SL}(2,\mathbb R),
$$
$a^*=\operatorname{diag}(a_1,\dots,a_n)$ and similarly for $b^*$, $c^*$ and $d^*$.

More generally (the following construction is taken from \cite{Fre79}), given $A \in\operatorname{GL}(n, \mathbb R)$, one can consider the map $\Delta^n \hookrightarrow\Omega$, given by $\Delta^n\ni z=(z_1,\dots, z_n) \mapsto A^tz^*A$. 
In order to take quotients, one defines
$$
\Gamma_A:=\left\{M \in\operatorname{SL}(2, \mathbb R)^n\quad\text{\rm such that}\quad 
\begin{pmatrix}
A^t & 0 \\
0 & A^{-1}
\end{pmatrix}
M^*
\begin{pmatrix}
A^t & 0 \\
0 & A^{-1}
\end{pmatrix}^{-1} \in \Gamma\right\}.
$$
Indeed, it is straightforward to see that we thus have a well--defined map
$$
\phi_A\colon \Delta^n / \Gamma_A \to X.
$$
 
Again following \cite{Fre90}, consider a totally real number field $\mathbb K$ of degree $n$ together with its embedding with dense image 
$$
\begin{aligned}
& \mathbb K \hookrightarrow\mathbb R^n \\
& \alpha \mapsto\bigl(\alpha^{(1)},\dots,\alpha^{(n)}\bigr).
\end{aligned}
$$
Next, fix a basis $\Lambda=\{\alpha_1,\dots,\alpha_n\}$ of $\mathbb K$ as a $\mathbb Q$-vector space. Then, the matrices 
$$
A_\Lambda=\bigl(\alpha_i^{(j)}\bigr)_{i,j=1,\dots,n}\in\operatorname{GL}(n,\mathbb K)
$$ 
have the property that the corresponding $\Gamma_{A_\Lambda}$ is commensurable with the Hilbert modular group associated to $\mathbb K$. Moreover, such matrices $A_\Lambda$ are clearly dense in $\operatorname{GL}(n,\mathbb R)$ as $\Lambda$ runs through all possible bases.

Now, take a global jet differential of order $k$ and wighted degree $m$ over $\overline X$, $P \in H^0\bigl(\overline{X},E_{k,m}^{GG}T_{\overline{X}}^*\bigr).$ Taking the pull--back $\phi_{A_\Lambda}^*P$, we obtain a $k$-jet differential on $\Delta^n/\Gamma_{A_\Lambda}$. Therefore, by Remark \ref{automorphic} and an evident variation in the non compact case of the subsequent proposition, $\phi_{A_\Lambda}^*P$ must vanish
on jets tangent to the directions given by the foliations defined by the factors. Thus, we have that $\phi_{A_\Lambda}(\Delta^n/\Gamma_A)\subset GG\bigl(\overline{X}\bigr)$. By density, we finally get $GG\bigl(\overline{X}\bigr)=\overline{X}$.
\end{proof}

\subsection*{The characteristic bundle {\cite[Chapter 6]{Mok89}}}

In order to \lq\lq globalize\rq\rq{} this polydisk approach, the right tool turns out to be the characteristic bundle introduced by Mok. We recall below one possible construction as well as some basic features of this important object.

So, let $\Omega=G_0/K$ be an irreducible bounded symmetric domain of rank $\ge 2$ endowed with its Bergman metric $\omega$ and $X=\Omega/\Gamma$ a quotient of $\Omega$ by a torsion--free discrete group of automorphisms. By a slight abuse of notation, we will still call $\omega$ the induced metric on $X$. Now, consider the projectivized bundle $\pi\colon P(T_X)\to X$ of lines of $T_X$ and the corresponding tautological line bundle $\mathcal O_{P(T_X)}(-1)\subset\pi^*T_X \to P(T_X)$ with the natural hermitian metric $h$ induced by $\omega$.

Next, let $\Theta_\omega(T_X)$ be the Chern curvature of $T_X$ with respect to the metric $\omega$. For each $x\in X$ and $v\in T_{X,x}\setminus\{0\}$, consider the hermitian form on $T_{X,x}$ defined by

\begin{equation}\label{hermitianform}
T_{X,x}\times T_{X,x}\ni(\xi,\eta)\mapsto \frac 1{||v||^2_\omega}\,\omega\bigl(\Theta_\omega(T_X)(v,\overline v)\cdot\xi,\eta\bigr).
\end{equation}

It is the hermitian form associated to the Griffiths curvature of $(T_X,\omega)$ in the direction given by $v$, which is therefore semi--negative. Call $\mathcal N_{x,[v]}\subset T_{X,x}$ its zero eigenspace: it is clearly neither zero nor the whole space. We say that $[v_0]$ is a \emph{characteristic direction} at the point $x\in X$ if $\dim \mathcal N_{x,[v_0]}$ is maximum among $\dim \mathcal N_{x,[v]}$, for $[v]\in T_{X,x}\setminus\{0\}$. This maximal dimension does depend only on $\Omega$: we call it $n(\Omega)$, the \emph{null dimension} of $\Omega$. 

\begin{definition}
The set $\mathcal S=\mathcal S(X)\subset P(T_X)$ of characteristic directions together with the induced projection $\pi|_\mathcal S\colon\mathcal S\to X$ onto $X$ is called the \emph{characteristic bundle}.
\end{definition}

The characteristic directions are, equivalently \cite[Proposition 1 on page 242]{Mok89}, the directions which minimize the holomorphic sectional curvature
$$
T_{X,x}\setminus\{0\}\ni v \mapsto \frac 1{||v||^4_\omega}\,\omega\bigl(\Theta_\omega(T_X)(v,\overline v)\cdot v,v\bigr).
$$
In algebraic terms, the characteristic vectors (\textsl{i.e.} the non zero vectors which define a characteristic direction) are given by the highest weights of the isotropy representation on the holomorphic tangent space.

It is a remarkable fact that $\pi|_\mathcal S\colon\mathcal S\to X$ is indeed a holomorphic fiber bundle and that $\mathcal S$ is a smooth closed complex submanifold of $P(T_X)$ of dimension
$$
\dim\mathcal S=\dim P(T_X)-n(\Omega)=2\dim X-1-n(\Omega).
$$
Even more remarkable, we have the following.

\begin{proposition}[Mok {\cite[Proposition 4 on page 262]{Mok89}}]\label{stablebl}
Let $X=\Omega/\Gamma$ be a compact quotient of an irreducible bounded symmetric domain of rank $\ge 2$. Then, for any integer $m>0$ and any $\sigma\in H^0\bigl(P(T_X),\mathcal O_{P(T_X)}(m)\bigr)$, we have that $\sigma$ vanishes identically on $\mathcal S$.
\end{proposition}

\begin{remark}
We stated here a weaker form of the proposition above, which gives in fact ---in its full strength--- a precise description of the stable base locus of $\mathcal O_{P(T_X)}(1)$ in terms of higher characteristic bundles, see \cite{Mok89} for more details.
\end{remark}

In the sequel, we shall need the following slightly refined version of Proposition \ref{stablebl}, which deals with global sections of the \emph{restriction} of the (anti)ta\-utological bundle to $\mathcal S$.

\begin{proposition}[Compare with {\cite[Chapter 6, (3.1) and (3.2)]{Mok89}}]\label{restrictedstablebl}
Let $X=\Omega/\Gamma$ be a compact quotient of an irreducible bounded symmetric domain of rank $\ge 2$. Then, for any integer $m>0$ we have
$$
H^0\bigl(\mathcal S,\mathcal O_{\mathcal S}(m)\bigr)=\{0\}, 
$$
where $\mathcal O_{\mathcal S}(m)$ is the restriction $\mathcal O_{P(T_X)}(m)|_\mathcal S$.

In other words, the Kodaira--Iitaka dimension of $\mathcal O_{\mathcal S}(1)$ is negative.
\end{proposition}
\begin{proof}
The proof relies essentially on the fact that the argument used to prove \cite[Proposition 4 on page 262]{Mok89} goes through in this stronger version. We sketch it anyway for the reader's convenience.

We proceed by contradiction and suppose there exists an integer $m>0$ and a non zero section $\sigma\in H^0\bigl(\mathcal S,\mathcal O_{\mathcal S}(m)\bigr)$. By a slight abuse of notation, we still call $h$ the restriction of the natural hermitian metric $h$ on $\mathcal O_{P(T_X)}(-1)$ to $\mathcal O_{\mathcal S}(-1)$. Then, $g:=(h^{m}+\sigma\otimes\overline\sigma)^{1/m}$ defines a hermitian metric on $\mathcal O_{\mathcal S}(-1)$ which still has semi--negative curvature, since $\sigma$ is holomorphic. 

\begin{claim}[Compare with Mok's hermitian metric rigidity theorem]\label{rigidity}
In this situation, $g=C\cdot h$ for some positive constant $C$. 
\end{claim}

This claim implies that $||\sigma||_{h^{-m}}$ is constant. If it were a non zero constant this would imply that $\sigma$ never vanishes and thus $\mathcal O_{\mathcal S}(m)$ is holomorphically trivial. But this is impossible since

\begin{claim}\label{integral}
The following inequality holds:
$$
\int_\mathcal S c_1\bigl(\mathcal O_{P(T_X)}(1)\bigr)\wedge\nu^{\dim\mathcal S-1}>0,
$$
where $\nu$ is the Kähler form on $P(T_X)$ given by $c_1\bigl(\mathcal O_{P(T_X)}(1),h^{-1}\bigr)+\pi^*\omega$. 
\end{claim}

But then $||\sigma||_{h^{-m}}\equiv 0$ and $\sigma$ is identically zero, too.
\end{proof}

We are left to prove the two claims above.

\begin{proof}[Proof of Claim \ref{rigidity}]
We follow almost word--by--word the alternative proof of Mok's hermitian metric rigidity theorem which makes use of Moore's ergodicity theorem \cite[Chapter 6, \textsection 3]{Mok89}: we just need to check that his argument goes through when everything lives on $\mathcal S$ and does not come necessarily from the whole $P(T_X)$.

So, write $g=e^u\cdot h$, where $u$ is a smooth real function on $\mathcal S$ and write $q:=n(\Omega)$. By \cite[Chapter 6, formula (2) on page 116]{Mok89}, we have that
$$
i\,\partial u\wedge\bar\partial u\wedge c_1\bigl(\mathcal O_{P(T_X)}(-1),h\bigr)^{2n-2q-1}\equiv 0\quad\text{\rm on $\mathcal S$}.
$$
This follows from a quite straightforward integral formula over $\mathcal S$. Now, we lift everything on $\Omega$, where the characteristic bundle is trivial 
$$
\mathcal S(\Omega)\simeq\Omega\times\mathcal S_0\subset\Omega\times\mathbb P^{n-1}=P(T_{\Omega}).
$$ 
Here, $\mathcal S_o$ is the set of characteristic direction over a base point $o\in\Omega$. In particular we have an identification 
$$
T_{\mathcal S(\Omega),(o,[v])}\simeq T_{\mathcal S_o,[v]}\oplus T_{\Omega,o}.
$$
From the explicit expression of the curvature $\Theta_h\bigl(\mathcal O_{P(T_X)}(-1)\bigr)$ it is immediate to check that the non--negative $(1,1)$-form $i\,\partial u\wedge\bar\partial u$ must vanish on $\mathcal N_{o,[v]}\subset T_{\Omega,o}$. Since $u$ is real, it follows that 
\begin{equation}\label{du}
du|_{\mathcal N_{o,[v]}}\equiv 0.
\end{equation} 

Let $U(T_X)$ the $\omega$-unitary tangent bundle to $X$ and $U(\mathcal S)\subset U(T_X)$ the subspace of unitary characteristic vectors. Of course, $u$ pulls--back to a ($S^1$-invariant) function on $U(\mathcal S)$, which we still call $u$ by abuse of notation. Fix a unit characteristic vector $v$ and let $L\subset G_0$ be its stabilizer, so that $U(\mathcal S)=\Gamma\backslash G_0/L$ has a locally homogeneous space (it is more convenient to we write $\Gamma$ on the left for this proof). By pulling--back again we get a positive $\Gamma$-invariant and $L$-invariant smooth function $\hat u$ on $G_0$ which we will show to be invariant under the right action of a closed non--compact subgroup $H\subset G_0$. Now, by Moore's ergodicity theorem, every $H$-invariant subset of $\Gamma\backslash G_0$ is either of zero or of full measure, since $G_0$ is simple and $\Gamma$ a lattice. If $\hat u\colon\Gamma\backslash G_0\to\mathbb R$ were not constant then we could find two positive real numbers $a<b$ such that the set 
$$
V_{a,b}:=\{y\in\Gamma\backslash G_0\mid a<\hat u(y)<b\}
$$
is neither of zero nor of full measure. But by the $H$-invariance of $\hat u$, $V_{a,b}$ would be $H$-invariant, contradiction. It then follows that $u$ itself is constant and therefore $g$ is a constant multiple of $h$.

We now come back to the existence of such a closed non--compact subgroup $H$: it is obtained as a one--parameter of transvections as follows. Take a non zero vector $w\in\mathcal N_{o,[v]}$ and let $\gamma$ be the geodesic on $\Omega$ determined by $w$. Without loss of generality, we can suppose that the characteristic vector $v$ is unitary and form the curve $\gamma^*\subset U(\mathcal S(\Omega))$ obtained by parallel transport of $v$ along $\gamma$ (here we use that $\mathcal S$ and hence $U(\mathcal S)$ are invariant by parallel transport). By construction, the curve $t\mapsto\gamma^*(t)$ is pointwise tangent to $\mathcal N_{\gamma(t),[\gamma^*(t)]}$, so that $u$ is constant along $\gamma^*$ by (\ref{du}). By the same argument, $u$ is also constant along the image $\varphi(\gamma^*)$ of $\gamma^*$ by any isometry $\varphi\in G_0$. By the general theory, $\gamma^*$ is the orbit $H(v)=H(e\mod L)$ of $v$ under a non--compact one--parameter family $H\subset G_0$ of isometries. But then, for $\varphi\in G_0$, $u$ is constant along the $H$-orbit $\varphi H\mod L=\phi(\gamma^*)$. This means exactly that $\hat u$ on $\Gamma\backslash G_0$ is invariant under the action of $H$.
\end{proof}

\begin{proof}[Proof of Claim \ref{integral}]
The Chern curvature $i\,\Theta_{h}\bigl(\mathcal O_{P(T_X)}(-1)\bigr)$ of $h$ at a point $(x,[v])\in P(T_X)$ is roughly given by a \lq\lq vertical\rq\rq{} term, which is nothing but minus the Fubini--Study metric, and a \lq\lq horizontal\rq\rq{} term, which is given by the $(1,1)$-form associated to the hermitian form in (\ref{hermitianform}). In particular, if $[v]$ is a characteristic direction, at $(x,[v])$ it is semi--negative of rank $2n-1-n(\Omega)$. So, the restriction of $i\,\Theta_{h}\bigl(\mathcal O_{P(T_X)}(-1)\bigr)$ to $\mathcal S$ is also semi--negative of rank at least $2n-1-2n(\Omega)>0$. Therefore, along $\mathcal S$, the trace with respect to $\nu$ 
$$
c_1\bigl(\mathcal O_{P(T_X)}(1),h^{-1}\bigr)\wedge\nu^{\dim\mathcal S-1}=-c_1\bigl(\mathcal O_{P(T_X)}(-1),h\bigr)\wedge\nu^{\dim\mathcal S-1}
$$
is strictly positive and thus its integral over $\mathcal S$, too.
\end{proof}

\subsection*{Foliating the characteristic bundle by minimal disks}

Let $\Omega=G_0/K$ be a bounded symmetric domain of rank $\ge 2$ and $X=\Omega/\Gamma$ a compact quotient. We shall recall how to construct a natural (smooth) holomorphic foliation by curves $\mathcal M$ on $\mathcal S$.
The foliation in question is the one associated to minimal disks in $\Omega$. Let us describe how one can construct it.
Let $\Omega^\vee=G_c/K$ be the compact dual of $\Omega$, where $G_c$ is a compact real form of the complexification of $G_0$. Write $\Omega\subset\!\subset\mathbb C^{\dim \Omega}\subset\Omega^\vee$ for the Harish--Chandra and Borel embeddings. As in the non compact case, one defines the notion of characteristic vectors on $\Omega^\vee$ as highest weights of the isotropy representations on the holomorphic tangent space. Since the isotropy representations for $\Omega$ and $\Omega^\vee$ are the same at points of $\Omega$, a tangent vector to $\Omega$ is characteristic for $\Omega$ if and only if it is characteristic for $\Omega^\vee$. Next, there is a notion of \emph{minimal rational curve} in $\Omega^\vee$: these are rational curves representing the positive generator of $H_2(\Omega^\vee,\mathbb Z)\simeq\mathbb Z$.

\begin{proposition}[Mok {\cite[Proposition 1 on page 143]{Mok89}}]
For any $x\in\Omega^\vee$ and any characteristic vector $v\in T_{\Omega^\vee,x}$ there exists a unique minimal rational curve $C$ passing through $x$ such that $T_{C,x}=[v]$. Furthermore, all minimal rational curves are obtained in this way.
\end{proposition}

Now, it can be shown that the intersection of any minimal rational curve with $\Omega$ is biholomorphic to a disk. Such disks in $\Omega$ are called \emph{minimal disks}. We shall still call minimal disk the image in $X$ of a minimal disk in $\Omega$ under the uniformization map $\Omega\to X$.

Consider the projectivized bundle $P(T_X)\to X$. For each non singular germ of holomorphic map $\varphi\colon(\mathbb C,0)\to X$, one can define a projectivized lifting $\widetilde\varphi\colon(\mathbb C,0)\to P(T_X)$ simply by writing $\widetilde\varphi(\zeta)=(\varphi(\zeta),[\varphi'(\zeta)])$; here, we think at point of $P(T_X)$ as pairs $(x,[v])$, where $x\in X$ and $v\in T_{X,x}\setminus\{0\}$. Tautologically, the restriction of $\mathcal O_{P(T_X)}(-1)$ to the image $\widetilde\varphi\bigl((\mathbb C,0)\bigr)$ of $\widetilde\varphi$ is isomorphic to the tangent bundle to $\widetilde\varphi\bigl((\mathbb C,0)\bigr)$. Now, as we saw, for each point $(x,[v])\in\mathcal S$ there exists a unique minimal disk $f\colon\Delta\to X$ passing through $x$ and whose tangent space at $x$ is $[v]$. Moreover, by uniqueness, for each $f(\zeta)$ one has that $(f(\zeta),[f'(\zeta)])\in\mathcal S$, \textsl{i.e.} $\widetilde f(\Delta)\subset\mathcal S$. The collection of all $\widetilde f(\Delta)$, when $f$ runs among all minimal disks in $X$, thus endows $\mathcal S$ with a holomorphic foliation by curve $\mathcal M$, whose tangent bundle $T_\mathcal M$ is naturally isomorphic to the restriction of $\mathcal O_{P(T_X)}(-1)$ to $\mathcal S$. 

\bigskip
The proof of Theorem \ref{locirredthm} is now almost immediate.

\begin{proof}[Proof of Theorem \ref{locirredthm}]
The idea is to use the characteristic bundle $\pi\colon\mathcal S\to X$ together with Remark \ref{relmorph}. 

The (smooth) holomorphic foliation $\mathcal M$ on $\mathcal S$ constructed above is such that the canonical bundle $K_\mathcal M$ to this foliation is isomorphic to $\mathcal O_{P(T_X)}(1)|_\mathcal S$. By Proposition \ref{restrictedstablebl}, $K_\mathcal M$ has negative Kodaira--Iitaka dimension. Thus, by Theorem \ref{main}, we obtain that $GG(\mathcal S,T_\mathcal M)=\mathcal S$.

Next, the tangent bundle $T_\mathcal M\subset T_\mathcal S\subset T_{P(T_X)}|_\mathcal S$ to $\mathcal M$ is transverse to the kernel of the differential $d\pi$. In fact given a direction tangent to $\mathcal M$, it is of the form $\widetilde f'$ for some minimal disk $f$ in $X$. Thus, since $\pi\circ\widetilde f=f$, we have
$$
d\pi\bigl(\widetilde f'\bigr)=f'\ne 0.
$$ 
Then, Remark \ref{relmorph} tells us that $\pi\bigl(GG(\mathcal S,T_\mathcal M)\bigr)\subset GG(X)$ and, since we have $GG(\mathcal S,T_\mathcal M)=\mathcal S$, Theorem \ref{locirredthm} follows.

\end{proof}
\section{Final remarks on the Green--Griffiths--Lang conjecture}

The final step for the solution of the Green--Griffiths--Lang conjecture using only (invariant) jet bundles requires to show that at least one of the base loci
$$
B_k:= \bigcap_{m\ge 1} \operatorname{Bs}\bigl(\mathcal O_{P_k(V)}(m)\otimes\pi_{0,k}^*A^{-1}\bigr)\setminus P_k(V)^{\textrm sing} \subset P_k(V)
$$ 
projects down to a proper algebraic subvariety $Y_k:=\pi_{0,k}(B_k)$ in $X$. These lines show that, unfortunately, this is hopeless without any further assumption on $X$ besides that of being of general type.

Nevertheless, one could think of a less demanding property, namely that for all irreducible subvarieties $Z \subset P_k(V)$ such that $\pi_{0,k}(Z)=X$, the restriction $\mathcal{O}_{P_k(V)}(1)|_Z$ could be big.

This is exactly what is proved in a theorem of Lu and Yau \cite{LY90} for order one jets, for $X$ a surface of general type with $c_1(X)^2-2\,c_2(X)>0$, achieving thus the proof of the Green--Griffiths--Lang conjecture in this case. But we can not expect even this less demanding property to be true in general since we have seen at the end of Section \ref{2} that there exist projective manifolds $X$ of general type such that for each $k>0$ there is a smooth submanifold $Z_k\subset P_k(V)$ which projects biholomorphically onto $X$ via $\pi_{0,k}$ and such that for all integers $m>0$ we have
$$
H^0\bigl(X_k,(\mathcal O_{P_k(V)}(m)\otimes\pi^*_{0,k}A^{-1})|_{X_k}\bigr)=0.
$$
Thus, we see that Lu and Yau's result is sharp in this sense and that $c_1^2(X)=2\,c_2(X)$ is somehow a threshold for projective surfaces of general type as far as the solution to the Green--Griffiths--Lang conjecture using only base locus of sections of $k$-jet bundles is concerned.

On the other hand, McQuillan's celebrated work on surfaces \cite{McQ98} shows that the conjecture is true whenever the second Segre number $c_1^2-c_2$ of the surface is positive (this being the case for the compact free quotients of the bidisc thanks to the Hirzebruch proportionality principle, see \cite{Hir58}). Observe that his proof relies upon a combination of the theories of jet differentials and that of foliations: jet differentials of order one (the existence of which is assured by the assumption on the Segre number) produce (multi)foliations such that every entire curve must be tangent to. The algebraic degeneracy is then obtained in this situation as a consequence of deep results about parabolic leaves of holomorphic foliations on surfaces of general type rather than trying to control the Green--Griffiths locus (which, \textsl{a posteriori} would not give the desired result, also in view of what is described in these lines).

As far as we know, at present, the only cases where an explicit control of the Green--Griffiths locus is known are given by generic projective hypersurfaces of high degree: it is shown in \cite{DMR10} that if $X\subset\mathbb P^{n+1}$, $n\ge 2$, is a generic projective hypersurface of degree $\deg X\ge 2^{n^5}$, then $GG(X)$ is contained in a proper subvariety of $X$. This result is further refined in \cite{DT10} where, under the same hypotheses, it is proved that $GG(X)$ is contained in a proper subvariety of codimension at least two. The reader may consult the survey \cite{dr11} for more details about hyperbolicity of projective hypersurfaces.

\bibliography{bibliography}{}

\end{document}